\documentclass[11pt,a4paper]{article}

\usepackage{latexsym}
\usepackage[T2A]{fontenc}
\usepackage[utf8x]{inputenc}
\usepackage[russian,english]{babel}

\usepackage{geometry}

\usepackage{graphicx}
\usepackage{longtable}
\usepackage{multirow}
\usepackage{amsfonts}
\usepackage{amsthm}
\usepackage{amsmath}
\usepackage{amssymb}
\usepackage{accents}
\usepackage{indentfirst}
\usepackage{enumerate}
\usepackage{float}
\usepackage{placeins}
\usepackage{mathtools}
\usepackage{csquotes}
\usepackage{setspace}
\usepackage{xifthen}
\usepackage{pdfpages}
\usepackage{transparent}
\usepackage{cite}
\usepackage{authblk}
\usepackage{arydshln}

\usepackage{rotating}
\usepackage{pdflscape}
\usepackage{multicol}

\def\A{\mathcal{A}}

\def\Main{\mathcal{M}}

\renewcommand{\Re}{\mathop{\mathfrak{Re}}\nolimits}
\renewcommand{\Im}{\mathop{\mathfrak{Im}}\nolimits}

\graphicspath{{./images/}}

\usepackage{xr-hyper}
\usepackage[colorlinks=true, allcolors=blue]{hyperref}

\newtheorem{theorem}{Theorem}[section]
\newtheorem{corollary}[theorem]{Corollary}
\newtheorem{proposition}[theorem]{Proposition}
\newtheorem{lemma}[theorem]{Lemma}
\theoremstyle{definition}
\newtheorem{definition}[theorem]{Definition}
\newtheorem{descript}[theorem]{Description}
\newtheorem{conjecture}[theorem]{Conjecture}

\newtheorem{notation}[theorem]{Notation}
\newtheorem{example}[theorem]{Example}
\newtheorem{remark}[theorem]{Remark}

\DeclareMathOperator{\sgn}{sgn}
\DeclareMathOperator{\Aut}{Aut}
\DeclareMathOperator{\Lin}{Lin}
\DeclareMathOperator{\Anc}{Anc}

\newcommand\til[1]{\widetilde{\phantom{,\mkern-4mu} #1 \phantom{,\mkern-4mu}}\mkern-1mu}

\newcommand\numeq[1]%
  {\stackrel{\scriptstyle(\mkern-1.5mu#1\mkern-1.5mu)}{=}}

\providecommand{\keywords}[1]{\textbf{Keywords:} #1}
\providecommand{\msc}[1]{\textbf{MSC 2020:} #1}

\newcommand\freefootnote[1]{%
    \bgroup
    \renewcommand\thefootnote{\fnsymbol{footnote}}%
    \renewcommand\thempfootnote{\fnsymbol{mpfootnote}}%
    \footnotetext[0]{#1}%
    \egroup
}

\begin{document}

\title{Relation graphs of the sedenion algebra}
\author{
Alexander Guterman$^{a,b,c}$ and Svetlana Zhilina$^{a,b}$
}
\date{\small \em
$^a$Department of Mathematics and Mechanics, Lomonosov Moscow State\\ University, Moscow, 119991, Russia\\
$^b$Moscow Center of Fundamental and Applied Mathematics, Moscow, 119991, Russia\\
$^c$Moscow Institute of Physics and Technology, Dolgoprudny, 141701, Russia
}

\maketitle

\begin{abstract}
Let $\mathbb{S}$ denote the algebra of the sedenions, and $\Gamma_O(\mathbb{S})$ denote its orthogonality graph. We observe that any pair of zero divisors in $\mathbb{S}$ produces a double hexagon in $\Gamma_O(\mathbb{S})$. The set of vertices of a double hexagon can be extended to a basis of $\mathbb{S}$ which has a convenient multiplication table. We describe explicitly the set of vertices of an arbitrary connected component of $\Gamma_O(\mathbb{S})$ and find its diameter. We then establish the bijection between the connected components of $\Gamma_O(\mathbb{S})$ and lines in the imaginary part of the octonions. Finally, we consider the commutativity graph of the sedenions and discover that all elements whose imaginary part is a zero divisor belong to the same connected component, and its diameter lies between $3$ and $4$.
\end{abstract}


\msc{05C25, 17A20, 17D05}

\freefootnote{This work was supported by the Russian Science Foundation (project No. 17-11-01124).}

\freefootnote{Email addresses: \texttt{alexander.guterman@biu.ac.il} (Alexander Guterman), \texttt{s.a.zhilina@gmail.com} (Svetlana Zhilina)}

\section{Introduction}

The algebra of the sedenions is an important example of real Cayley-Dickson algebras. It is the first algebra of the main sequence which is not alternative, and, therefore, has zero divisors. The structure of its zero divisors proves to be interesting and harmonious.

Some of the most prominent articles on zero divisors of the algebras of the main sequence include the works by Moreno~\cite{moreno, moreno_constructing} and Biss, Dugger, and Isaksen~\cite{biss, biss2}. Moreno was the first to study doubly alternative zero divisors, that is, the elements whose both components are alternative in the previous algebra of this sequence. He established several interesting properties of such zero divisors, see~\cite[pp. 25-27]{moreno}.

A convenient method of visualization of a binary algebraic relation $R$ is to define a corresponding graph. Its vertices represent elements or their equivalence classes in an algebraic structure under consideration, and there is an edge from $x$ to $y$ if and only if $xRy$. The most popular relation graphs of various algebras are commutativity, orthogonality, and zero divisor graphs.

Some of the recent works on relation graphs of real Cayley-Dickson algebras include~\cite{our_split-algebras, our_split-sedenions} where relation graphs of the split-complex numbers, the split-quaternions, the split-octonions, and the split-sedenions have been described. We refer to these algebras as real low-dimensional Cayley-Dickson split-algebras. Moreover,~\cite[Theorem~4.30]{our_split-algebras} establishes a simple relationship between commutativity and orthogonality graphs of these algebras. In case of the sedenions, however, this relationship is much more complex. The reason is that we should use Lemma~\ref{lemma:A_n-commutativity-through-orthogonality}\eqref{item:norm-zero} of the current paper for low-dimensional split-algebras, while Lemma~\ref{lemma:A_n-commutativity-through-orthogonality}\eqref{item:norm-nonzero} holds for the sedenions.

In the work~\cite{our_orthographs1} by the second author we have studied zero divisors with some alternativity and norm restrictions in arbitrary real Cayley-Dickson algebras. We have discovered that they form hexagonal patterns in orthogonality and zero divisor graphs, see~\cite[Corollary~3.7]{our_orthographs1}.
In the algebras of the main sequence, any pair of zero divisors generates the so-called double hexagon, see~\cite[Description~4.8]{our_orthographs1}.
The vertices of a double hexagon have a convenient multiplication table which has a block structure, cf.~\cite[Theorem~4.11]{our_orthographs1}.

In this paper we study the orthogonality graph of the sedenions, denoted by $\Gamma_O(\mathbb{S})$. It is well known that the algebra of the octonions is alternative, and every sedenion can be expressed as a pair of octonions. Hence, in case of the sedenions, every zero divisor is doubly alternative, so we may apply the results of~\cite{our_orthographs1} and obtain that every pair of zero divisors generates a double hexagon in $\Gamma_O(\mathbb{S})$. 

One can verify that any automorphism $\phi$ of the octonions can be extended to the automorphism $\widehat{\phi}$ of the sedenions by setting $\widehat{\phi}((a,b)) = (\phi(a), \phi(b))$. Khalil and Yiu have shown that the group $\Aut_{\mathbb{R}}(\mathbb{O})$ acts freely and transitively on
$$
\{ (x,y) \in \mathbb{S} \times \mathbb{S} \; | \; n(x) = n(y) = 1, \, xy = 0 \}.
$$
In particular, any pair of zero divisors $(a,b)$ and $(c,d)$ can always be replaced with some multiples of $(e_1,e_2)$ and $(e_7,e_4)$. It has been proved in~\cite[Proposition~3.4(ii)]{chan} that, since $e_1 e_2 = e_3 = - e_7 e_4$, the condition $(a,b)(c,d) = 0$ implies $n(c)ab = -n(a)cd$. We will see soon that this property of the sedenions is particularly important. Specifically, it secures that the set of vertices of a double hexagon can be extended to a basis which has a neat multiplication table, cf. Theorem~\ref{theorem:neat-sedenions-basis}.

We would like to mention the work by Chan and {\DJ}okovi\'c~\cite{chan} which is devoted to the classification of the subalgebras of $\mathbb{S}$. When choosing the representatives of their conjugacy classes under the action of $\Aut_{\mathbb{R}}(\mathbb{S})$, Chan and {\DJ}okovi\'c definitely used some relations between zero divisors which are stated in our multiplication table. Conversely, the multiplication table from this work provides an easy method to write down several subalgebras of $\mathbb{S}$ which contain an arbitrary zero divisor, see Proposition~\ref{proposition:zero-divisor-subalgebras}.

We also study the connected components of $\Gamma_O(\mathbb{S})$. Clearly, for any two zero divisors $(a,b)$ and $(c,d)$ which belong to the same connected component, the elements $ab$ and $cd$ are linearly dependent. It follows also from~\cite[pp.~25-27]{moreno} that $ab$ and $cd$ are pure. Hence we can assign a unique line in the imaginary part of the octonions to each connected component of $\Gamma_O(\mathbb{S})$. By Theorem~\ref{theorem:component-octonion}, this correspondence is actually bijective. The diameter of each connected component equals~$3$, see Theorem~\ref{theorem:sedenions-orthogonality}.

In Subsection~\ref{subsection:basis-pairs} we describe a subgraph of $\Gamma_O(\mathbb{S})$ on those vertices whose both components are standard basis elements up to sign, that is, which are of the form $(e_i, \pm e_j)$. We note that the hexagons in Figure~\ref{figure:subgraphs} have already appeared in a work by Brown~\cite[Theorem~8.1]{brown}. Besides, zero divisors of the form $(e_i, \pm e_j)$ have been studied by de Marrais~\cite{marrais}. He has also obtained an analogue of our double hexagons~\cite[p.~3]{marrais2} and the multiplication table of their vertices~\cite[p.~8]{marrais3}. We emphasize, however, that our proof is more general.

In~\cite{cawagas} Cawagas has classified all subloops of the loop of standard base sedenions. He has found $7$ isomorphic copies of a quasi-octonionic loop and has discovered that all zero divisors found by de Marrais are confined to these copies.

Finally, in Section~\ref{section:commutativity-graph} we consider the commutativity graph of $\mathbb{S}$. We show that all elements whose imaginary part is a zero divisor belong to the same connected component, and its diameter lies between $3$ and $4$. Based on several computations in Wolfram Mathematica, we conjecture that the diameter is likely to be equal to $3$.

\section{An overview of real Cayley-Dickson algebras} \label{section:A_n}

\subsection{Algebraic relations and their graphs} \label{subsection:definitions}

Let $\mathbb{F}$ be an arbitrary field and $(\A, +, \cdot)$ be an algebra with a unity $1_{\A}$ over the field $\mathbb{F}$. $\A$ is not assumed to be commutative or associative. We say that $a, b \in \A$ {\em anticommute} if $ab + ba = 0$, and $a$ and $b$ {\em are orthogonal} if $ab = ba = 0$.
We denote the set of zero divisors (left, right, or two-sided) of $\A$ by $Z(\A)$, and the set of two-sided zero divisors of $\A$ by $Z_{LR}(\A)$. We also denote the center of $\A$ by $C_{\A}$.

We denote the set of all automorphisms of an algebra $\A$ over a field $\mathbb{F}$ by $\Aut_{\mathbb{F}}(\A)$. Evidently, any automorphism preserves pairs of commuting elements, pairs of orthogonal elements, and pairs of zero divisors.

\begin{definition}
Let $a$ be an arbitrary element of $\A$.
\begin{itemize}
    \item
    {\em The centralizer} of $a$ is $C_\A(a) = \big\{ b \in \A \: | \: ab=ba \big\}$, namely, the set of all elements in $\A$ which commute with $a$.
    \item
    {\em The anticentralizer} of $a$ is $\Anc_\A(a) = \big\{ b \in \A \: | \:  ab+ba=0 \big\}$, namely, the set of all elements in $\A$ which anticommute with $a$.
    \item
    {\em The orthogonalizer} of $a$ is $O_\A(a)=\big\{ b \in \A \: | \;  ab=ba=0 \big\}$, namely, the set of all elements in $\A$ which are orthogonal to $a$.
\end{itemize}
\end{definition}

Clearly, $C_\A(a)$, $\Anc_\A(a)$, and $O_\A(a)$ are vector spaces over $\mathbb{F}$.

\begin{notation}
For any set $X \subseteq \A$ we denote the set of lines passing through elements of $X$ by
$$
P(X) = \{ [x] = \mathbb{F} x \; | \; x \in X \}.
$$
\end{notation}

We can now introduce some relation graphs to be studied in this paper.

\begin{definition}
Let $\A$ be an arbitrary algebra.
\begin{itemize}
\item 
{\em The commutativity graph} $\Gamma_C(\A)$ is defined as follows: its vertices are lines in $\A \setminus C_\A$, that is,
$$
V(\Gamma_C(\A)) = P(\A \setminus C_\A),
$$
and distinct vertices $[a]$ and $[b]$ are adjacent if and only if $ab=ba$.
\item 
{\em The orthogonality graph} $\Gamma_O(\A)$ is defined as follows: its vertices are lines in $Z_{LR}(\A)$, that is,
$$
V(\Gamma_O(\A)) = P(Z_{LR}(\A)),
$$
and distinct vertices $[a]$ and $[b]$ are adjacent if and only if $ab=ba=0$.
\end{itemize}
\end{definition}

Note that the edges of $\Gamma_C(\A)$ and $\Gamma_O(\A)$
are well-defined. When speaking of the vertices of these graphs, we will not distinguish between a nonzero element $a$ and a line $[a] = \mathbb{F} a$ passing through it.

Recall also from the graph theory that, in an undirected graph $\Gamma$, $d(x,y) = d_{\Gamma}(x,y)$ denotes the distance between two vertices $x$ and $y$, and $d(\Gamma) = \sup\limits_{x,y \in \Gamma} d(x,y)$ denotes the diameter of $\Gamma$. A clique is a set of pairwise adjacent vertices, and it is called maximal if it is is maximal by inclusion. The girth $g(\Gamma)$ is the length of the shortest cycle in $\Gamma$.

\subsection{Constructing Cayley-Dickson algebras} \label{subsection:A_n}

We refer the reader to~\cite{mccrimmon,schafer} for auxiliary definitions and general properties of Cayley-Dickson algebras.

\begin{definition} \label{definition:cayley-dickson-algebras}
Let $\A$ be an algebra over a field $\mathbb{F}$ with an involution $a \mapsto \bar{a}$. The algebra $\A \{ \gamma \}$ produced by the Cayley-Dickson process, when applied to $\A$ with the parameter $\gamma \in \mathbb{F}$, $\gamma \neq 0$, is defined as the set of ordered pairs of elements of $\A$ with operations
\begin{align*}
\alpha(a,b)&=(\alpha a, \alpha b);\\
(a,b)+(c,d)&=(a+c,b+d);\\
(a,b)(c,d)&=(ac+\gamma \bar{d}b,da+b\bar{c})
\end{align*} 
and the involution
$$
\qquad (\overline{a,b})=(\bar{a},-b), \qquad a,b,c,d\in \A, \ \alpha \in \mathbb{F}.
$$
If the involution on $\A$ is regular, that is, $a + \bar{a} \in \mathbb{F}1_{\A}$ and $a\bar{a} = \bar{a}a \in \mathbb{F}1_{\A}$ for all $a \in \A$, then the involution on $\A \{ \gamma \}$ is also regular, cf.~\cite[p.~435]{schafer}.
\end{definition}

Henceforth we assume that $\mathbb{F} = \mathbb{R}$. We now define an arbitrary real Cayley-Dickson algebra which is determined by the set of its parameters. The most common definition of real Cayley-Dickson algebras assumes that all parameters are equal to $-1$. These particular algebras are what we call the algebras of the main sequence.

\begin{definition} \label{definition:A_n}
For every integer $n \geq 0$ and nonzero real numbers $\gamma_0, \dots, \gamma_{n-1}$ we define the real Cayley-Dickson algebra $\A_n = \A_n \{ \gamma_0, \dots, \gamma_{n-1} \}$ inductively:
\begin{enumerate} [(1)]
\item $\A_0 = \mathbb{R}$, and $e^{(0)}_0 = 1$ is its only basis element;
\item If $\A_n \{ \gamma_0, \dots, \gamma_{n-1} \}$ is constructed then $\A_{n+1} \{ \gamma_0, \dots, \gamma_n \}=(\A_n \{ \gamma_0, \dots, \gamma_{n-1} \}) \{ \gamma_n \}$. Its basis elements are $e^{(n+1)}_0, \dots, e^{(n+1)}_{2^{n+1}-1}$ such that
\begin{equation*}
e^{(n+1)}_m =
\begin{cases}
(e^{(n)}_m,0), & 0 \leq m \leq 2^n-1,\\
(0,e^{(n)}_{m-2^n}), & 2^n \leq m \leq 2^{n+1}-1.
\end{cases}
\end{equation*}
\end{enumerate}
\end{definition}

For every integer $n \geq 0$ the structure $\A_n$ in Definition~\ref{definition:A_n} is a $2^n$-dimensional algebra over $\mathbb{R}$ with the unit element $e^{(n)}_0$ and a regular involution, cf.~\cite[Lemma 3.14]{our_split-algebras}. We will denote $1=e^{(n)}_0$ and $r=re^{(n)}_0$ for $r \in \mathbb{R}$. Consider the following definitions which are analogous to those for complex numbers.

\begin{definition} \label{definition:real-imaginary-part}
\leavevmode
\begin{itemize}
	\item Let $a \in \A_n$. Its {\em real part} is $\Re(a) = \frac{a + \bar{a}}{2}$, its {\em imaginary part} is $\Im(a) = \frac{a - \bar{a}}{2}$, and its {\em norm} is $n(a) = a \bar{a} = \bar{a}a$. Since the involution on $\A_n$ is regular, $\Re(a), n(a) \in \mathbb{R}$.
	\item An element $a \in \A_n$ is said to be {\em pure} if $\Re(a)=0$.
	\item An element $(a, b) \in \A_{n+1}$ is said to be {\em doubly pure} if $\Re(a) = \Re(b) = 0$.
\end{itemize}
\end{definition}

\begin{proposition} {\rm \cite[p. 435]{schafer}}
We can compute real part and norm of an element $(a,b) \in \A_{n+1}$ inductively by using the following equalities:
\begin{align*}
	\Re((a,b)) &= \Re(a),\\
	n((a,b)) &= n(a) - \gamma_n n(b).
\end{align*}
\end{proposition}

\begin{remark}
The norm of $a$ is often defined as $\sqrt{a\bar{a}}$, unlike $n(a)=a\bar{a}$ in this paper. However, most of the results can be easily extended to the norm modified in this way.
\end{remark}

\subsection{Some properties of real Cayley-Dickson algebras} \label{subsection:A_n-properties}

Henceforth we assume that $\A$ is an arbitrary algebra over a field $\mathbb{F}$, and $\A_n = \A_n \{ \gamma_0, \dots, \gamma_{n-1} \}$ is an arbitrary real Cayley-Dickson algebra. By~\cite[Exercise 2.5.1]{mccrimmon}, $\A_n \{ \gamma_0, \dots, \gamma_{n-1} \}$ is isomorphic to $\A_n \{ \sgn(\gamma_0), \dots, \sgn(\gamma_{n-1}) \}$, so it is sufficient to consider only $\gamma_k \in \{ \pm 1 \}$, $k = 0, \dots, n-1$.

\begin{notation} \label{proposition:lambda-form}
Let $\langle a, b \rangle$ denote a real-valued symmetric bilinear form associated with the quadratic form $n(a)$. Then $\langle a, a \rangle = n(a)$ and $2\langle a,b \rangle = a \bar{b} + b \bar{a} = \bar{a} b + \bar{b} a$ for all $a, b \in \A_n$, cf. \cite[Proposition~3.18, Proposition~3.19]{our_split-algebras}.

We remark that for all $a, b \in \A_n$ we have $\langle a, b \rangle = \langle \bar{a}, \bar{b} \rangle$ and $\Re(a) = \langle a, 1 \rangle$.
\end{notation}

The following lemma describes the anticentralizer of an arbitrary nonzero pure element of $\A_n$.

\begin{lemma} {\rm \cite[Lemma 5.8]{our_anticomm} } \label{lemma:A_n-anticomm}
Let $a \in \A_n$, $\Re(a) = 0$, $a \neq 0$. Then 
\begin{equation*}
    \Anc_{\A_n}(a) = \left\{ b \in \A_n \; | \; \Re(b) = 0 \mbox{ and } \langle a,b \rangle = 0 \right\}.
\end{equation*}
\end{lemma}

We now proceed to some concepts related to associativity. For $a,b,c \in \A$ we denote their associator by $[a,b,c] = (ab)c - a(bc)$, and their anti-associator by $\{ a,b,c \} = (ab)c + a(bc)$. An element $a \in \A$ is called {\em alternative} if for all $b \in \A$ the equalities $[a,a,b] = 0$ and $[b,a,a] = 0$ hold. If all elements of $\A$ are alternative then $\A$ is called {\em alternative}. It is well known that $\A_n$ is alternative if and only if $n \leq 3$. However, all Cayley-Dickson algebras are {\em flexible}, that is, they satisfy $[a,b,a] = 0$ for all $a,b \in \A$, cf.~\cite[p.~437, Theorem~1]{schafer}.

\begin{proposition} \label{proposition:skew-symm-associator}   \rm \cite[Exercise 2.1.1]{mccrimmon}
\leavevmode
\begin{itemize}
    \item
    If $\A$ is flexible then for all $a,b,c \in \A$ we have $[a,b,c]=-[c,b,a]$.
    \item
    If $\A$ is alternative then the associator in $\A$ is skew-symmetric, that is, it changes sign if an argument transposition is performed.
\end{itemize}
\end{proposition}

\begin{notation}
Let $m \in \mathbb{N}$, $a_1, \dots, a_m \in \A_n$. Then we denote
\begin{align*}
\Lin(a_1, \dots, a_m) &= \mathbb{R}a_1 + \dots + \mathbb{R}a_m,\\
\Lin^*(a_1, \dots, a_m) &= \Lin(a_1, \dots, a_m) \setminus \{ 0 \}.
\end{align*}
\end{notation}

\section{Algebras of the main sequence} \label{section:main-sequence}

\subsection{Definition and examples} \label{subsection:A_n-examples}

\begin{definition}
It is said that the algebra $\A_n \{ \gamma_0, \dots, \gamma_{n-1} \}$ is an algebra of {\em the main sequence}, if $\gamma_k = -1$ for each $k = 0, \dots, n-1$. We denote this algebra by $\Main_n$.
\end{definition}

\begin{proposition} {\rm \cite[Proposition~3.31]{our_split-algebras}} \label{proposition:A_n-euclidean-product}
Let $a = \sum\limits_{m=0}^{2^n-1} a_{m}e^{(n)}_{m}, \ \ b = \sum\limits_{m=0}^{2^n-1} b_{m}e^{(n)}_{m} \in \Main_n$. Then $\langle a, b \rangle = \sum\limits_{m=0}^{2^n-1} a_m b_m$ is a Euclidean inner product. Particularly, $n(a) = \sum\limits_{m=0}^{2^n-1} a_{m}^2$, so $n(a)=0$ if and only if $a=0$.
\end{proposition}

\begin{example}
Some examples of the real Cayley-Dickson algebras of the main sequence include the complex numbers ($\mathbb{C}$), the quaternions ($\mathbb{H}$), the octonions ($\mathbb{O}$), and the sedenions ($\mathbb{S}$) for $n=1,\:2,\:3,$ and $4$, correspondingly. We refer the reader to~\cite{baez} for the definitions of these algebras.
\end{example}

Exact definitions and some basic properties of $\mathbb{O}$ and $\mathbb{S}$ are given below.

\begin{definition} {\rm \cite[p. 6]{baez}}
$\mathbb{O}$ is an eight-dimensional algebra over $\mathbb{R}$, its basis elements being equal to $1,e_1,\dots,e_7$. The involution in $\mathbb{O}$ is given by the formula
$$
\overline{a_0 + a_1e_1 + \dots + a_7e_7} = a_0 - a_1e_1 - \dots - a_7e_7,
$$
and multiplication is given by Table~\ref{table:octonions-mult}.

\begin{table}[H]
\centering
$
\begin{array}{|c|cccccccc|}
\hline
\times&1&e_1&e_2&e_3&e_4&e_5&e_6&e_7\\\hline
1&1&e_1&e_2&e_3&e_4&e_5&e_6&e_7\\
e_1&e_1&-1& e_3& -e_2& e_5& -e_4& -e_7& e_6\\
e_2&e_2&-e_3& -1& e_1& e_6& e_7& -e_4& -e_5\\
e_3&e_3&e_2& -e_1& -1& e_7& -e_6& e_5& -e_4\\
e_4&e_4&-e_5& -e_6& -e_7& -1& e_1& e_2& e_3\\
e_5&e_5&e_4& -e_7& e_6& -e_1& -1& -e_3& e_2\\
e_6&e_6&e_7& e_4& -e_5& -e_2& e_3& -1& -e_1\\
e_7&e_7&-e_6& e_5& e_4& -e_3& -e_2& e_1& -1\\\hline
\end{array}
$
\caption{\label{table:octonions-mult} Multiplication table of the unit octonions.}
\end{table}
\end{definition}

It is well known that $\mathbb{O} \cong \Main_3$, and thus $\mathbb{O}$ is a non-commutative non-associative but alternative algebra without zero divisors, cf.~\cite[p.~10]{baez}.

The algebra of {\em the sedenions} is defined as $\mathbb{S} = \Main_4$. One can easily see that $\mathbb{S} = \Main_4 = \Main_3 \{ -1 \} \cong \mathbb{O} \{ -1 \}$, and this isomorphism provides the most convenient representation for the sedenions. $\mathbb{S}$ is a non-commutative non-associative and non-alternative algebra with zero divisors, cf.~\cite[p.~2]{moreno}.

\subsection{Zero divisors and double hexagons} \label{subsection:main-sequence}

By~\cite[Corollary~4.6]{our_split-algebras}, in case of real Cayley-Dickson algebras all zero divisors appear to be two-sided zero divisors, that is, $Z(\A_n) = Z_{LR}(\A_n)$. In case of the algebras of the main sequence a stronger result holds.

\begin{remark}
By~\cite[Corollary~1.6]{moreno}, $xy = 0$ in $\Main_n$ if and only if $yx = 0$. Hence $\Gamma_O(\Main_n)$ completely represents relations between zero divisors in $\Main_n$.
\end{remark}

\begin{lemma} {\rm \cite[Corollary~1.12]{moreno}} \label{lemma:moreno-tilde}
Let $x = (x_1,x_2)$, $y = (y_1,y_2)$, $\til{y} = (-y_2,y_1) \in \Main_n$. Then $xy=0$ if and only if $x\til{y} = 0$.
\end{lemma}

\begin{definition} {\rm \cite[p. 15]{moreno_alternative}}
Let $a, b \in \A_n$.
\begin{itemize}
    \item We say that $a$ alternates with $b$ if $[a,a,b] = 0$.
    \item We say that $a$ alternates strongly with $b$ if $[a,a,b] = 0$ and $[b,b,a] = 0$.
\end{itemize}
\end{definition}

It is well known that any zero divisor in $\Main_n$ is doubly pure, cf.~\cite[Corollary~1.9]{moreno}. However, alternativity restrictions give us more conditions on zero divisors. The following lemma appeared first in Moreno's work~\cite[pp. 25-27]{moreno}, and then was generalized in~\cite{our_orthographs1}.

\begin{lemma} {\rm \cite[Corollary~4.4]{our_orthographs1}}
\label{lemma:A_n-strongly-alternative-properties}
Let $a,b \in \Main_{n-1}$ alternate strongly with $c,d \in \Main_{n-1}$, $(a,b),(c,d) \in Z(\Main_n)$, $(a,b)(c,d) = 0$. Then
\begin{enumerate}[(1)]
\item $\Re(a)=\Re(b)=\Re(c)=\Re(d)=0$;
\item $n(a)=n(b)$, $n(c)=n(d)$;
\item $[a,c,b] = 2n(a)d$, $[a,d,b] = -2n(a)c$, $[c,a,d] = 2n(c)b$, $[c,b,d] = -2n(c)a$;
\item $\{ a,c,b \} = \{ a,d,b \} = \{ b,c,a \} = \{ b,d,a \} = 0$,\\
$\{ c,a,d \} = \{ c,b,d \} = \{ d,a,c \} = \{ d,b,c \} = 0$.
\end{enumerate}
If $a,b,c,d$ alternate strongly pairwise then
\begin{enumerate}[(1)]
\setcounter{enumi}{4}
\item $a,b,c,d$ are pairwise orthogonal with respect to the inner product $\langle \cdot, \cdot \rangle$;
\item $ac = bd$, $ad = -bc$.
\end{enumerate}
\end{lemma}

\begin{theorem} {\rm \cite[Lemma~3.2, Lemma~4.6, Corollary~4.7, Description~4.8, Remark~4.13]{our_orthographs1}}
\label{theorem:double-hexagon}
Let $a,b \in \Main_{n-1}$ alternate strongly with $c,d \in \Main_{n-1}$, $(a,b),(c,d) \in Z(\Main_n)$, $(a,b)(c,d) = 0$. By Lemma~\ref{lemma:A_n-strongly-alternative-properties}, we may assume without loss of generality that $n(a) = n(b) = n(c) = n(d) = 1$. Then 
\begin{enumerate}[(1)]
    \item The elements $ac,ad$ alternate strongly with $a,b,c,d$. \label{item:strong-alternativity}
    \item $1,a,b,c,d,ac,ad$ form an orthonormal system with respect to the inner product $\langle \cdot, \cdot \rangle$. \label{item:orthonormal-system}
    \item There exists a subgraph of $\Gamma_O(\Main_n)$ which is depicted in Figure~\ref{figure:double-hexagon}. We call it a {\em double hexagon}. \label{item:double-hexagon}
    \item All elements in the vertices of this double hexagon are linearly independent. \label{item:linear-independence}
\end{enumerate}
\end{theorem}

\begin{remark} \label{remark:next-pair}
In particular, it follows from Theorem~\ref{theorem:double-hexagon}\eqref{item:double-hexagon} that $(a,b)(c,d) = 0$ implies $(c,d)(ac,ad) = 0$.
\end{remark}

\begin{figure}[H]
\centering
\includegraphics[width=0.55\linewidth]{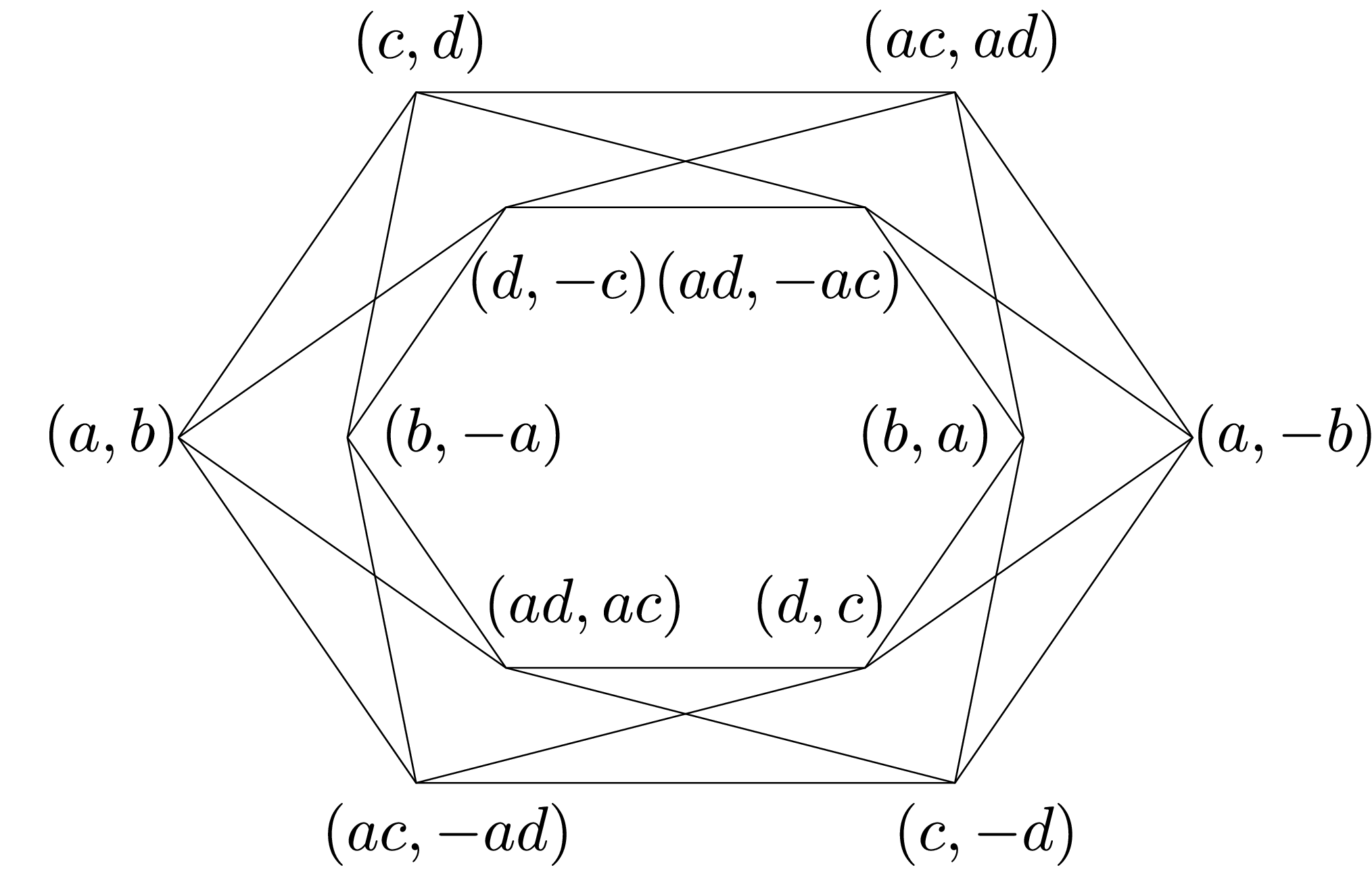}
\caption{\label{figure:double-hexagon} A double hexagon.}
\end{figure}

\section{Orthogonality graph of the sedenions} \label{section:sedenions}

\subsection{Zero divisors and their properties} \label{subsection:sedenion-zero-divisors}

We now consider the case when $\Main_n = \mathbb{S}$. Since $\mathbb{O}$ is alternative, any pair of zero divisors $(a,b),(c,d) \in \mathbb{S}$, $(a,b)(c,d) = 0$, satisfies the conditions of Lemma~\ref{lemma:A_n-strongly-alternative-properties}. In particular, we may always assume that $n(a) = n(b) = n(c) = n(d) = 1$. Then, by Corollary~2.12 and Theorem~2.13~in~\cite{moreno}, there exists an automorphism $\phi$ of $\mathbb{O}$ which maps $a,b,c,d$ to $e_1,e_2,e_7,e_4$, respectively. We extend $\phi$ to the automorphism of $\mathbb{S}$ by the formula $(x,y) \mapsto (\phi(x),\phi(y))$. Then $(a,b)$ and $(c,d)$ are mapped to $(e_1,e_2)$ and $(e_7,e_4)$. Hence we can always replace $(a,b)$ and $(c,d)$ with $(e_1,e_2)$ and $(e_7,e_4)$. We obtain immediately the following proposition.

\begin{proposition} \label{proposition:sedenions-orthonormal-system}
Let $(a,b),(c,d) \in Z(\mathbb{S})$, $(a,b)(c,d) = 0$,  $n(a) = n(b) = n(c) = n(d) = 1$. Then $1,a,b,c,d,ab,ac,ad$ form an orthonormal system with respect to the inner product $\langle \cdot, \cdot \rangle$.
\end{proposition}

In case of the sedenions, we have one more relation between the components of pairs of zero divisors compared to Lemma~\ref{lemma:A_n-strongly-alternative-properties}. In other algebras, this relation takes a weaker form and holds for base units only.

\begin{lemma} \label{lemma:sedenions-zero-divisors-relation} \label{lemma:basis-pairs-zero-divisors-relation} {\rm \cite[Proposition~3.4(ii)]{chan}, \cite[Lemma~6.2]{our_orthographs1}}
\begin{enumerate}[(1)]
    \item Let $(a,b)(c,d)=0$ in $\mathbb{S}$, $n(a) = n(b) = n(c) = n(d) = 1$. Then $ab = -cd$.
    \item Let $n \geq 1$, $a,b,c,d \in \{ \pm e_m^{(n-1)} \: | \: m = 1, \dots 2^{n-1} - 1\}$, $(a,b),(c,d) \in Z(\A_n)$, $(a,b)(c,d) = 0$. Then $ab = \pm cd$.
\end{enumerate}
\end{lemma}

\begin{example} \label{example:zero-divisors-relation}
Let $(a,b) = (e_1, e_2), (c,d) = (e_5 + e_{13}, e_6 + e_{14}) \in \Main_5$. Then $(a,b)(c,d) = 0$. Moreover, it follows from~\cite[Theorem~3.3]{moreno_alternative} that $a,b,c,d$ are alternative elements in $\Main_4 = \mathbb{S}$. However, $ab = e_3$, $cd = 2e_{11}$, so $ab$ and $cd$ are linearly independent.
\end{example}

\begin{corollary} \label{corollary:sedenions-path}
Let $(a,b), (c,d) \in Z(\mathbb{S})$, $n(a) = n(b) = n(c) = n(d) = 1$. Let also $P$ be a path of length $k$ in $\Gamma_O(\mathbb{S})$ from $(a,b)$ to $(c,d)$. Then $ab = (-1)^k cd$.
\end{corollary}

\begin{proof}
Assume without loss of generality that for any inner element $(x,y) \in P$ we have $n(x) = n(y) = 1$. Then the statement follows from Lemma~\ref{lemma:sedenions-zero-divisors-relation} by induction on~$k$.
\end{proof}

\begin{proposition} \label{proposition:odd-cycles}
\leavevmode
\begin{enumerate}[(1)]
    \item $\Gamma_O(\mathbb{S})$ contains no cycles of odd length.
    \item The girth of $\Gamma_O(\mathbb{S})$ equals 4.
    \item Each maximal clique in $\Gamma_O(\mathbb{S})$ consists of 2 vertices.
\end{enumerate}
\end{proposition}

\begin{proof}
\leavevmode
\begin{enumerate}[(1)]
    \item Assume from the contrary that there exists a path $P$ which starts and finishes in $(a,b) \in Z(\mathbb{S})$, $n(a) = n(b) = 1$, and the length of $P$ is an odd number $k$. By Corollary~\ref{corollary:sedenions-path}, we have $ab = (-1)^k ab = -ab$. Hence $ab = 0$, however, $\mathbb{O}$ contains no zero divisors, a contradiction.
    \item We first show that $g(\Gamma_O(\mathbb{S})) \leq 4$. Let $(a,b)(c,d)=0$ where $(a,b), (c,d) \in Z(\mathbb{S})$. Then, by Lemma~\ref{lemma:moreno-tilde}, there exists the following $4$-cycle in $\Gamma_O(\mathbb{S})$:
    $$
    (a,b) \leftrightarrow (c,d) \leftrightarrow (b,-a) \leftrightarrow (d,-c) \leftrightarrow (a,b).
    $$

    Since there are no cycles of odd length in $\Gamma_O(\mathbb{S})$, $g(\Gamma_O(\mathbb{S})) \geq 4$. Thus $g(\Gamma_O(\mathbb{S})) = 4$.
    \item There are no $3$-cycles in $\Gamma_O(\mathbb{S})$. Moreover, $\Gamma_O(\mathbb{S})$ contains no isolated vertices. Thus we obtain the statement required.
\end{enumerate}
\end{proof}

\subsection{Connected components of the orthogonality graph}

In the statements~\ref{proposition:hexagon-basis}~-~\ref{corollary:components-vertex-set} we assume that $(a,b),(c,d) \in Z(\mathbb{S})$, $(a,b)(c,d) = 0$, $n(a) = n(b) = n(c) = n(d) = 1$. Then $(a,b),(c,d)$ satisfy the conditions of Theorem~\ref{theorem:double-hexagon}, and they are contained in a double hexagon from Figure~\ref{figure:double-hexagon}.

\begin{proposition} \label{proposition:hexagon-basis}
Every element of the double hexagon is adjacent to exactly $4$ other elements of this hexagon which form a basis of the orthogonalizer of this element. Particularly,
\begin{align*}
    O_{\mathbb{S}}((a,b)) &= \Lin((c,d), (d,-c), (ac,-ad), (ad,ac)),\\
    O_{\mathbb{S}}((c,d)) &= \Lin((a,b), (b,-a), (ac,ad), (ad,-ac)),\\
    O_{\mathbb{S}}((ac,ad)) &= \Lin((c,d), (d,-c), (a,-b), (b,a)).
\end{align*}
\end{proposition}

\begin{proof}
The inclusion from right to left follows immediately from Theorem~\ref{theorem:double-hexagon}\eqref{item:double-hexagon}. The converse inclusion is clear from~\cite[p. 25]{moreno}, since the orthogonalizer of an arbitrary zero divisor in $\mathbb{S}$ has dimension four.
\end{proof}

\begin{notation}
Let us denote
\begin{align*}
\Lambda^+_{(a,b)} &= \Lin^*((a,b), (b,-a), (d,c), (c,-d), (ac,ad), (ad,-ac)),\\
\Lambda^-_{(a,b)} &= \Lin^*((b,a), (a,-b), (c,d), (d,-c), (ad,ac), (ac,-ad)),\\
\Lambda_{(a,b)} &= \Lambda^+_{(a,b)} \cup \Lambda^-_{(a,b)}.
\end{align*}
\end{notation}

$\Lambda_{(a,b)}$ is the set of all nontrivial linear combinations of the elements of the every other angle of the double hexagon in Figure~\ref{figure:double-hexagon}.

\begin{lemma} \label{lemma:components}
$\Lambda_{(a,b)}$ lies at a distance at most $3$ from $(a,b)$ in $\Gamma_O(\mathbb{S})$.
\end{lemma}

\begin{proof}
Consider some $(x,y) \in O_{\mathbb{S}}((a,b))$, $(x,y) = k_1(c,d) + k_2(d,-c) + k_3(ad,ac) + k_4(ac,-ad)$. Then $(ax,ay) = k_1(ac,ad) + k_2(ad,-ac) - k_3(d,c) - k_4(c,-d)$. By Remark~\ref{remark:next-pair}, $(x,y)(ax,ay) = 0$. Moreover, it follows from Figure~\ref{figure:double-hexagon} that $(a,b), (b,-a) \in O_\mathbb{S}((x,y))$ and $(b,a), (a,-b) \in O_\mathbb{S}((ax,ay))$. Hence for any $k_5, k_6 \in \mathbb{R}$ such that the last element of the path is nonzero we have the following paths of length at most $3$:
\begin{itemize}
    \item $(a,b) \longleftrightarrow (x,y) \longleftrightarrow k_5(a,b) + k_6(b,-a)$;
    \item $(a,b) \longleftrightarrow (x,y) \longleftrightarrow (ax,ay) + k_5(a,b) + k_6(b,-a)$;
    \item $(a,b) \longleftrightarrow (x,y) \longleftrightarrow (ax,ay) \longleftrightarrow k_5(b,a) + k_6(a,-b)$;
    \item $(a,b) \longleftrightarrow (x,y) \longleftrightarrow (ax,ay) \longleftrightarrow (x,y) + k_5(b,a) + k_6(a,-b)$.
\end{itemize}
The statement of the lemma follows immediately.
\end{proof}

\begin{lemma} \label{lemma:components-vertex-set}
Let $(x,y) \in Z(\mathbb{S})$ be such that $xy \in \Lin(ab)$. Then $(x,y) \in \Lambda_{(a,b)}$.
\end{lemma}

\begin{proof}
By Proposition~\ref{proposition:sedenions-orthonormal-system}, the elements $1,a,b,c,d,ab,ac,ad$ form an orthonormal system with respect to the inner product $\langle \cdot, \cdot \rangle$. Assume without loss of generality that $n(x) = n(y) = 1$. Similarly, we obtain that $1,x,y,xy$ form an orthonormal system, so $xy = \pm ab$ and $x, y \in \Lin(a,b,c,d,ac,ad)$. Then if we choose $x \in \Lin(a,b,c,d,ac,ad)$, $n(x) = 1$, the second component $y = (\bar{x}x)y = -x(xy) = \mp x(ab)$ is determined up to sign. Thus we have $(x,y) \in \Lambda_{(a,b)}$.
\end{proof}

\begin{corollary} \label{corollary:components-vertex-set}
Let $C$ denote the connected component of $\Gamma_O(\mathbb{S})$ which contains $(a,b)$. Then the vertex set of $C$ is $P(\Lambda_{(a,b)})$.
\end{corollary}

\begin{proof}
The inclusion from right to left is proved in Lemma~\ref{lemma:components}. Let now $(x,y) \in C$. By Corollary~\ref{corollary:sedenions-path}, $xy \in \Lin(ab)$. Then Lemma~\ref{lemma:components-vertex-set} implies that $(x,y) \in \Lambda_{(a,b)}$.
\end{proof}

\begin{theorem} \label{theorem:sedenions-orthogonality}
The diameter of each connected component of $\Gamma_O(\mathbb{S})$ equals 3.
\end{theorem}

\begin{proof}
Let $(a,b) \in Z(\mathbb{S})$. Then there exists $(c,d) \in Z(\mathbb{S})$ such that $(a,b)(c,d)=0$. Assume without loss of generality that $n(a) = n(b) = n(c) = n(d) = 1$. Let us denote the connected component of $\Gamma_O(\mathbb{S})$ which contains $(a,b)$ and $(c,d)$ by $C$. By Corollary~\ref{corollary:components-vertex-set}, the vertex set of $C$ is $P(\Lambda_{(a,b)})$. Furthermore, it follows from Lemma~\ref{lemma:components} that $\Lambda_{(a,b)}$ lies at a distance at most $3$ from $(a,b)$ in $\Gamma_O(\mathbb{S})$.

Note that Theorem~\ref{theorem:double-hexagon}\eqref{item:linear-independence} implies that $O_\mathbb{S}((a,b)) \cap O_\mathbb{S}((b,a)) = 0$, so $d((a,b),(b,a)) = 3$.

Since $(a,b)$ is arbitrary, the diameter of each connected component equals 3.
\end{proof}

\begin{lemma} {\rm \cite[Proposition 12.1]{biss}} \label{lemma:sedenions-zero-divisor-pairs}
Let $a, b \in \mathbb{O}$, $(a,b) \in \mathbb{S} \setminus \{ 0 \}$. Then $(a,b) \in Z(\mathbb{S})$ if and only if the following conditions are satisfied:
\begin{enumerate}[(1)]
\item $n(a) = n(b)$;
\item $1,a,b$ are orthogonal with respect to the inner product $\langle \cdot, \cdot \rangle$.
\end{enumerate}
\end{lemma}

\begin{lemma} {\rm \cite[Lemma 1.3]{moreno}, \cite[Lemma 4.2, Lemma 4.8]{our_split-algebras}} \label{lemma:inner-product-movement} \label{lemma:alternative-elements-are-normed}
\begin{enumerate}[(1)]
    \item Let $a, b, c \in \A_n$. Then $\langle a, bc \rangle = \langle a\bar{c}, b \rangle = \langle \bar{b}a, c \rangle$. \label{item:inner-product-movement}
    \item Let $a, b \in \A_n$, $a$ alternates with $b$. Then $n(ab) = n(ba) = n(a)n(b)$. \label{item:alternative-elements-are-normed}
\end{enumerate}
\end{lemma}

\begin{notation}
Let $\Gamma$ be an undirected graph. Then $\mathcal{C}(\Gamma)$ denotes the set of all connected components of~$\Gamma$.
\end{notation}

\begin{theorem} \label{theorem:component-octonion}
Let $P(\Im(\mathbb{O}))$ denote the set of all one-dimensional subspaces in $\Im(\mathbb{O})$. Then there exists a well-defined bijection $\psi: \mathcal{C}(\Gamma_O(\mathbb{S})) \rightarrow P(\Im(\mathbb{O}))$ which acts as follows. Let $C \in \mathcal{C}(\Gamma_O(\mathbb{S}))$, $(a,b) \in C$. Then $\psi(C) = \Lin(ab)$.
\end{theorem}

\begin{proof}
By Corollary~\ref{corollary:sedenions-path}, $\psi$ is well-defined. By Lemma~\ref{lemma:components-vertex-set}, $\psi$ is injective. We now show that $\psi$ is surjective.

Consider some $X \in P(\Im(\mathbb{O}))$. Then $X = \Lin(x)$ for some $x \in \Im(\mathbb{O})$, $n(x) = 1$. Let now $a \perp \Lin(1,x)$, $n(a) = 1$, $b = \bar{a}x$. Then $ab = a(\bar{a}x) = (a\bar{a})x = x$. Lemma~\ref{lemma:sedenions-zero-divisor-pairs} implies that $(a,b) \in Z(\mathbb{S})$:
\begin{enumerate}[(1)]
    \item by Lemma~\ref{lemma:alternative-elements-are-normed}\eqref{item:alternative-elements-are-normed}, $n(b) = n(\bar{a})n(x) = n(a)$;
    \item it follows from Lemma~\ref{lemma:inner-product-movement}\eqref{item:inner-product-movement} that $\langle 1, b \rangle = \langle a, x \rangle = 0$, $\langle a, b \rangle = \langle a^2, x \rangle = - \langle 1, x \rangle = 0$, so $1, a, b$ are orthogonal.
\end{enumerate}
Let $C \in \mathcal{C}(\Gamma_O(\mathbb{S}))$ be such that $(a,b) \in C$. Then $\psi(C) = X$.
\end{proof}

\begin{corollary}
Let $\phi \in \Aut_{\mathbb{R}}(\mathbb{O})$ generate an automorphism $\widehat{\phi} \in \Aut_{\mathbb{R}}(\mathbb{S})$ such that $\widehat{\phi}((x,y)) = (\phi(x), \phi(y))$ for all $(x,y) \in \mathbb{S}$. Then $\widehat{\phi}$ acts naturally on $\mathcal{C}(\Gamma_O(\mathbb{S}))$, and $\psi \circ \widehat{\phi} = \phi \circ \psi$.
\end{corollary}

\begin{proof}
Note that $\widehat{\phi}$ preserves pairs of orthogonal elements and thus maps connected components of $\Gamma_O(\mathbb{S})$ into connected components. Let $C \in \mathcal{C}(\Gamma_O(\mathbb{S}))$, $(a,b) \in C$. Then $(\phi(a), \phi(b)) = \widehat{\phi}((a,b)) \in \widehat{\phi}(C)$, so $\psi(\widehat{\phi}(C)) = \Lin(\phi(a) \phi(b)) = \Lin(\phi(ab)) = \phi(\Lin(ab)) = \phi(\psi(C))$. 
\end{proof}

\subsection{The subgraph of $\Gamma_O(\mathbb{S})$ on the elements of the form $\left[\left(e_i, \pm e_j\right)\right]$} \label{subsection:basis-pairs}

\begin{notation}
We will need the following subsets of $Z(\mathbb{S})$:
\begin{align*}
Z_e &= \left\{ \left(e_i, \pm e_j\right) \in Z(\mathbb{S}) \right\},\\
Z_e^+ &= \left\{ \left(e_i, e_j\right) \in Z(\mathbb{S}) \right\},\\
Z_e^- &= \left\{ \left(e_i, -e_j\right) \in Z(\mathbb{S}) \right\}.
\end{align*}

Then $\Gamma_e(\mathbb{S})$ denotes a subgraph of $\Gamma_O(\mathbb{S})$ on the vertex set $P(Z_e)$.
\end{notation}

\begin{descript}
$\Gamma_e(\mathbb{S})$ is depicted in Figure~\ref{figure:subgraphs} based on~\cite{cawagas,marrais}. 

$\Gamma_e(\mathbb{S})$ is disconnected and consists of 7 double hexagons. For convenience, we assume that outer elements belong to $Z_e^+$, while inner ones belong to $Z_e^-$. Let double hexagons in Figure~\ref{figure:subgraphs} be numbered $1, \dots, 7$ from left to right and from top to bottom. We denote the $k$th double hexagon by $H_k$.
\end{descript}

\begin{figure}
\centering
\includegraphics[width=\linewidth]{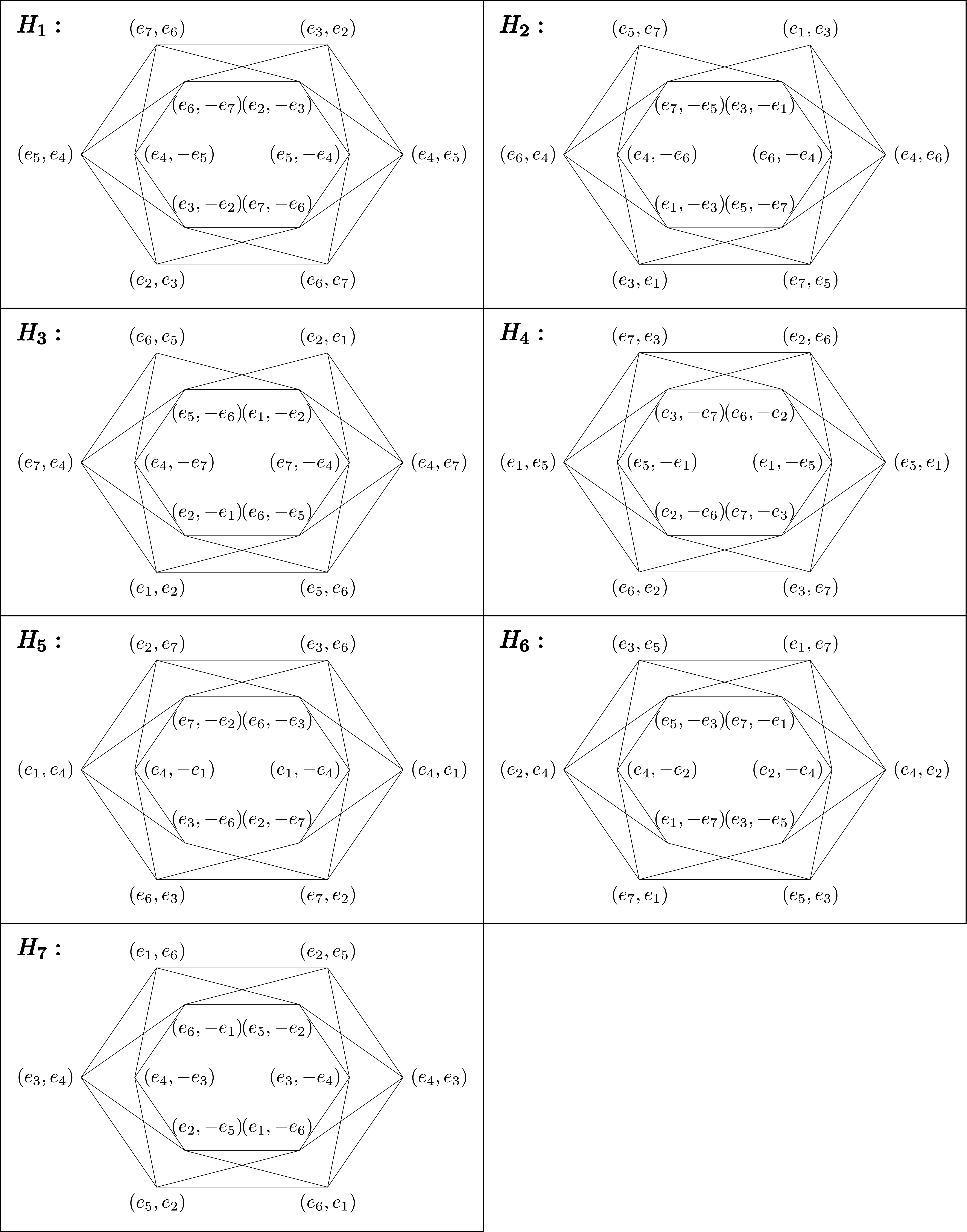}
\caption{\label{figure:subgraphs} The subgraph of $\Gamma_O(\mathbb{S})$ on the elements of the form $\left[\left(e_i, \pm e_j\right)\right]$.}
\end{figure}

\begin{proposition} \label{proposition:graph-simple-properties}
There are several properties of $\Gamma_e(\mathbb{S})$ which are listed below:
\begin{enumerate}[(1)]
    \item $(e_i, \pm e_j) \in Z(\mathbb{S})$ if and only if $i \neq j$;
    
    \item $(e_i,e_j)(e_k,e_l)=0$ if and only if $(e_i,-e_j)(e_k,-e_l)=0$,
    
    $(e_i,e_j)(e_k,-e_l)=0$ if and only if $(e_i,-e_j)(e_k,e_l)=0$;
    
    \item For all $k \in \{ 1, \dots, 7 \}$ $H_k$ is indeed a double hexagon from Figure~\ref{figure:double-hexagon}, rearranged so that outer elements belong to $Z_e^+$ and inner ones belong to $Z_e^-$;
    
    \item Each double angle of a double hexagon contains a pair of elements $(e_i,e_j)$ and $(e_j,-e_i)$. The opposite one contains the elements $(e_j,e_i)$ and $(e_i,-e_j)$;
    
    \item Let $(e_i, \pm e_j) \in H_k$. Then $e_i e_j = \pm e_k$. \label{item:component-octonion}
\end{enumerate}
\end{proposition}

\begin{proof}
\leavevmode
\begin{enumerate}[(1)]
    \item Follows from Lemma~\ref{lemma:sedenions-zero-divisor-pairs}.
    
    \item Follows immediately from Theorem~\ref{theorem:double-hexagon}.
    
    \item One can verify this by direct calculations.
    
    \item Follows from the structure of the double hexagon in Figure~\ref{figure:double-hexagon}.
    
    \item Let $(e_i, e_j)$ be an arbitrary element within $H_k$. Then it follows from Theorem~\ref{theorem:component-octonion} that for all $(e_l, e_m) \in H_k$ we have $e_l e_m \in \Lin^*(e_i e_j)$. Hence $e_l e_m = \pm e_i e_j$. By renumbering double hexagons we can easily obtain $e_i e_j = \pm e_k$ for all $(e_i, \pm e_j) \in H_k$.
\end{enumerate}
\end{proof}

\begin{proposition}
Distinct connected components of $\Gamma_e(\mathbb{S})$ are contained in distinct connected components of $\Gamma_O(\mathbb{S})$.
\end{proposition}

\begin{proof}
Follows immediately from Theorem~\ref{theorem:component-octonion} and Proposition~\ref{proposition:graph-simple-properties}\hyperref[item:component-octonion]{(\ref{item:component-octonion})}.
\end{proof}

\section{The multiplication table of the vertices of a double hexagon}

\begin{lemma} \label{lemma:mult-product}
Let $1, x, y$ be an orthonormal system in $\mathbb{O}$. Then we have in $\mathbb{S}$
\begin{align*}
    (x,y)(0,xy) &= -(x,-y),\\
    (x,y)(xy,0) &= -(y,x).
\end{align*}
\end{lemma}

\begin{proof}
By Lemma~\ref{lemma:A_n-anticomm}, the elements $x$ and $y$ anticommute. Since $x$ and $y$ are pure, we have $x^2 = -n(x) = -1$ and $y^2 = -n(y) = -1$. By the alternativity of $\mathbb{O}$, we obtain
\begin{align*}
    (x,y)(0,xy) &= (-\overline{(xy)}y, (xy)x) = ((xy)y, -(yx)x) = (x(yy), -y(xx)) = -(x,-y),\\
    (x,y)(xy,0) &= (x(xy), y\overline{(xy)}) = (x(xy), y(yx)) = ((xx)y, (yy)x) = -(x,y).
\end{align*}
\end{proof}

Let $(a,b),(c,d) \in Z(\mathbb{S})$, $(a,b)(c,d) = 0$, $n(a) = n(b) = n(c) = n(d) = 1$. Then $(a,b),(c,d)$ satisfy the conditions of Theorem~\ref{theorem:double-hexagon}, so Figure~\ref{figure:double-hexagon} represents a double hexagon in $\Gamma_O(\mathbb{S})$ which contains $(a,b)$ and $(c,d)$. The vertices of this hexagon can be completed to a basis of $\mathbb{S}$ which has a neat multiplication table.

\begin{theorem} \label{theorem:neat-sedenions-basis}
Let
\begin{align*}
    f_0 &= (1,0), & f_4 &= (ac,-ad), & \til{f_0} &= (0,1), & \til{f_4} &= (ad,ac),\\
    f_1 &= (a,b), & f_5 &= (c,d), & \til{f_1} &= (-b,a), & \til{f_5} &= (-d,c),\\
    f_2 &= (c,-d), & f_6 &= (a,-b), & \til{f_2} &= (d,c), & \til{f_6} &= (b,a),\\
    f_3 &= (ac,ad), & f_7 &= (0,ab), & \til{f_3} &= (-ad,ac), & \til{f_7} &= (-ab,0).
\end{align*}
Then $F = \{ f_m, \til{f_m} \: | \: m = 0, \dots, 7 \}$ is a basis in $\mathbb{S}$ which is orthogonal with respect to the inner product $\langle \cdot, \cdot \rangle$. Its multiplication table is given in Table~\ref{table:sedenions-mult}.
\end{theorem}

\begin{proof}
We have $f_0 = e_0$ and $\til{f_0} = e_8$ in the standard basis of $\mathbb{S}$. By Proposition~\ref{proposition:sedenions-orthonormal-system}, the elements $1,a,b,c,d,ab,ac,ad$ form an orthonormal system with respect to the inner product $\langle \cdot, \cdot \rangle$. Hence $F$ is an orthogonal system in $\mathbb{S}$. By Lemma~\ref{lemma:A_n-anticomm}, the elements of $F \setminus \{ f_0 \}$ anticommute pairwise.

It follows from~\cite[p.~13]{moreno} that for any $(x,y) \in \mathbb{S}$ we have $(x,y)\til{f_0} = (x,y)(0,1) = (-y,x)$. Hence $f_m \til{f_0} = \til{f_m}$ and $\til{f_m} \til{f_0} = -f_m$ for all $m = 0, \dots, 7$.

Since $(a,b)(c,d) = 0$ and $(c,d)(ac,ad) = 0$, it follows from Lemma~\ref{lemma:sedenions-zero-divisors-relation} that $ab = -cd = (ac)(ad)$. The elements $a,b,c,d,ac,ad$ anticommute pairwise, so we have
$$
ab = -ba = dc = -cd = (ac)(ad) = -(ad)(ac).
$$
Then the multiplication table for $\{ f_m, \til{f_m} \: | \: m = 0, \dots, 6 \}$ follows from~\cite[Theorem~4.11]{our_orthographs1}.
We add columns and rows corresponding to $f_7$ and $\til{f_7}$ by using Lemma~\ref{lemma:mult-product}.
\end{proof}

\begin{remark}
According to Lemma~\ref{lemma:moreno-tilde}, the elements $f_m$ and $-\til{f_m}$ belong to the same corner of a double hexagon, $m = 1, \dots, 6$. The element $f_m$ belongs to the outer hexagon, while $-\til{f_m}$ belongs to the inner hexagon.

The elements $f_1$, $f_2$ and $f_3$ belong to every other angle of the outer hexagon, while $f_4$, $f_5$ and $f_6$ belong to its remaining angles. Each of the pairs $f_1$ and $f_6$, $f_2$ and $f_5$, $f_3$ and $f_4$ corresponds to the pair of opposite vertices of the outer hexagon.
\end{remark}

Table~\ref{table:sedenions-mult} has a block structure, its $(3 \times 3)$-blocks being of two types. Some of them are antidiagonal, and the others resemble the multiplication table of the unit quaternions. On the whole, Table~\ref{table:sedenions-mult} reminds us of the multiplication table for the standard basis of $\mathbb{S}$, cf.~\cite[p.~254]{cawagas}.
Our basis $F$ has been found earlier by de Marrais in~\cite[p.~8]{marrais3}, however, his ordering of the elements differs from ours, so his multiplication table has a different structure and consists of $(4 \times 4)$-blocks.

Let us recall the work by Chan and {\DJ}okovi\'c~\cite{chan} which is devoted to the classification of the subalgebras of $\mathbb{S}$. An orbit under the action of $\Aut_{\mathbb{R}}(\mathbb{S})$ on the set of all subalgebras of $\mathbb{S}$ is called a conjugacy class. The isomorphism of two subalgebras $S_1, S_2 \subseteq \mathbb{S}$ is denoted by $S_1 \cong S_2$, and their conjugacy is denoted by $S_1 \sim S_2$. It is shown in~\cite{chan} that the dimension of an arbitrary subalgebra of $\mathbb{S}$ has to be $1,2,3,4,6$ or $8$. Moreover, conjugacy classes for each of these dimensions are described. When defining the representatives of these conjugacy classes, Chan and {\DJ}okovi\'c used some relations between zero divisors which are stated in Table~\ref{table:sedenions-mult}, however, they did not write this table explicitly. Let
\begin{align*}
    S_3 &= \Lin((1,0), (e_1,e_4), (e_2,e_7)), & \mathbb{O}_{i,j,e} &= \mathbb{H} + \mathbb{H}(0,1),\\
    S_4 &= \Lin((1,0), (e_1,e_2), (e_4,e_7), (e_5,-e_6)), & R_8 &= S_4 + S_4(0,1),\\
    S_6 &= S_3 + S_3(0,1), & S_8 &= \mathbb{H} + \mathbb{H}(0,e_4). 
\end{align*}
Chan and {\DJ}okovi\'c~\cite{chan} have shown that:
\begin{itemize}
    \item Any $3$-dimensional subalgebra is conjugate to $S_3$;
    \item Any $4$-dimensional subalgebra is either isomorphic to $\mathbb{H}$ or conjugate to $S_4$;
    \item Any $6$-dimensional subalgebra is conjugate to $S_6$;
    \item Any $8$-dimensional subalgebra is conjugate to one of the subalgebras $\mathbb{O}, \mathbb{O}_{i,j,e}, R_8$ or $S_8$. The subalgebra $S_8$ is the only non-division algebra among them. The algebras $\mathbb{O}$ and $\mathbb{O}_{i,j,e}$ are isomorphic, while $R_8$ is not isomorphic to them.
\end{itemize}

The following proposition describes some subalgebras of $\mathbb{S}$ arising from Table~\ref{table:sedenions-mult}. We also classify them according to the results of~\cite{chan}.

\begin{proposition} \label{proposition:zero-divisor-subalgebras}
We list some nontrivial subalgebras of $\mathbb{S}$ which contain $f_1$.
\begin{tabbing}
\hspace{18mm}\=\kill
$\dim = 3\!:$ \> $\Lin(f_0,f_1,f_4) \sim S_3$, $\Lin(f_0,f_1,f_5) \sim S_3$,\\
\> $\Lin(f_0,f_1,\til{f_4}) \sim S_3$, ${\Lin(f_0,f_1,\til{f_5}) \sim S_3}$.\\
$\dim = 4\!:$ \> $\Lin(f_0,f_1,f_6,f_7) \cong \mathbb{H}$, $\Lin(f_0,f_1,\til{f_0},\til{f_1}) \cong \mathbb{H}$,\\
\> $\Lin(f_0,f_1,\til{f_6},\til{f_7}) \cong \mathbb{H}$, $\Lin(f_0,f_1,f_2,f_3) \sim S_4$.\\
$\dim = 6\!:$ \> $\Lin(f_0,f_1,f_4,\til{f_0},\til{f_1},\til{f_4}) \sim S_6$, $\Lin(f_0,f_1,f_5,\til{f_0},\til{f_1},\til{f_5}) \sim S_6$.\\
$\dim = 8\!:$ \> $\Lin(f_0,f_1,f_2,f_3,f_4,f_5,f_6,f_7) \sim S_8$, $\Lin(f_0,f_1,f_6,f_7,\til{f_2},\til{f_3},\til{f_4},\til{f_5}) \sim S_8$,\\
\> $\Lin(f_0,f_1,f_6,f_7,\til{f_0},\til{f_1},\til{f_6},\til{f_7}) \sim \mathbb{O}_{i,j,e}$, $\Lin(f_0,f_1,f_2,f_3,\til{f_0},\til{f_1},\til{f_2},\til{f_3}) \sim R_8$.
\end{tabbing}
\end{proposition}

\begin{proof}
One can verify by using Table~\ref{table:sedenions-mult} that these vector spaces are closed under multiplication, so they are subalgebras of $\mathbb{S}$.

The conjugacy to $S_3$ and $S_6$ when $\dim = 3$ and $\dim = 6$ follows from~\cite[Theorem~8.1]{chan}. 

The isomorphism between $\Lin(f_0,f_1,f_6,f_7)$, $\Lin(f_0,f_1,\til{f_0},\til{f_1})$, $\Lin(f_0,f_1,\til{f_6},\til{f_7})$ and $\mathbb{H}$ is clear if we normalize $f_1, f_6, \til{f_1}, \til{f_6}$, that is, divide these vectors by $\sqrt{2}$. However, the subalgebra $\Lin(f_0,f_1,f_2,f_3)$ is not associative, and thus not isomorphic to $\mathbb{H}$, see~\cite[p.~496]{chan}. Hence $\Lin(f_0,f_1,f_2,f_3) \sim S_4$.

Since $f_1f_4 = f_1\til{f_4} = 0$, $\Lin(f_0,f_1,f_2,f_3,f_4,f_5,f_6,f_7)$ and $\Lin(f_0,f_1,f_6,f_7,\til{f_2},\til{f_3},\til{f_4},\til{f_5})$ are $8$-dimensional non-division subalgebras, so they are conjugate to $S_8$.

We have $\Lin(f_0,f_1,f_6,f_7) \cong \mathbb{H}$, so $f_1$ and $f_6$ alternate strongly in $\mathbb{S}$. Moreover, $f_6 \perp \Lin(f_0,f_1,\til{f_0},\til{f_1})$. Hence, by~\cite[Lemma~5.8]{our_orthographs1},
$\Lin(f_0,f_1,f_6,f_7,\til{f_0},\til{f_1},\til{f_6},\til{f_7})$ is isomorphic to $\mathbb{O}$. Since the set $\{ \pm \til{f_0} \}$ is invariant under $\Aut_{\mathbb{R}}(\mathbb{S})$, see~\cite[Lemma~2.1]{eakin}, and $\til{f_0} \notin \mathbb{O}$, the subalgebra $\Lin(f_0,f_1,f_6,f_7,\til{f_0},\til{f_1},\til{f_6},\til{f_7})$ is not conjugate to $\mathbb{O}$. Hence it is conjugate to $\mathbb{O}_{i,j,e}$.

Finally, $\Lin(f_0,f_1,f_2,f_3,\til{f_0},\til{f_1},\til{f_2},\til{f_3}) = \Lin(f_0,f_1,f_2,f_3) + \Lin(f_0,f_1,f_2,f_3) (0,1)$. We have $\Lin(f_0,f_1,f_2,f_3) \sim S_4$, so $\Lin(f_0,f_1,f_2,f_3,\til{f_0},\til{f_1},\til{f_2},\til{f_3}) \sim S_4 + S_4(0,1) = R_8$.
\end{proof}

\begin{remark}
Since $f_1 = (a,b)$ is an arbitrary zero divisor in $\mathbb{S}$, Proposition~\ref{proposition:zero-divisor-subalgebras} provides subalgebras which contain an arbitrary zero divisor of $\mathbb{S}$.
\end{remark}

\begin{landscape}
\begin{table}[H]
\centering
$
\begin{array}{|c||c:ccc:ccc:c:c:ccc:ccc:c|}
\hline
\vphantom{\Big|} \times & f_0 & f_1 & f_2 & f_3 & f_4 & f_5 & f_6 & f_7 & \til{f_0} & \til{f_1} & \til{f_2} & \til{f_3} & \til{f_4} & \til{f_5} & \til{f_6} & \til{f_7}\\
\hline\hline
\vphantom{\Big|} f_0 & f_0 & f_1 & f_2 & f_3 & f_4 & f_5 & f_6 & f_7 & \til{f_0} & \til{f_1} & \til{f_2} & \til{f_3} & \til{f_4} & \til{f_5} & \til{f_6} & \til{f_7}\\
\hdashline
\vphantom{\Big|} f_1 & f_1 & -2 f_0 & 2 f_3 & -2 f_2 & 0 & 0 & 2 f_7 & -f_6 & \til{f_1} & -2 \til{f_0} & -2 \til{f_3} & 2 \til{f_2} & 0 & 0 & -2 \til{f_7} & \til{f_6}\\
\vphantom{\Big|} f_2 & f_2 & -2 f_3 & -2 f_0 & 2 f_1 & 0 & 2 f_7 & 0 & -f_5 & \til{f_2} & 2 \til{f_3} & -2 \til{f_0} & -2 \til{f_1} & 0 & -2 \til{f_7} & 0 & \til{f_5}\\
\vphantom{\Big|} f_3 & f_3 & 2 f_2 & -2 f_1 & -2 f_0 & 2 f_7 & 0 & 0 & -f_4 & \til{f_3} & -2 \til{f_2} & 2 \til{f_1} & -2 \til{f_0} & -2 \til{f_7} & 0 & 0 & \til{f_4}\\
\hdashline
\vphantom{\Big|} f_4 & f_4 & 0 & 0 & -2 f_7 & -2 f_0 & -2 f_6 & 2 f_5 & f_3 & \til{f_4} & 0 & 0 & 2 \til{f_7} & -2 \til{f_0} & 2 \til{f_6} & -2 \til{f_5} & -\til{f_3}\\
\vphantom{\Big|} f_5 & f_5 & 0 & -2 f_7 & 0 & 2 f_6 & -2 f_0 & -2 f_4 & f_2 & \til{f_5} & 0 & 2 \til{f_7} & 0 & -2 \til{f_6} & -2 \til{f_0} & 2 \til{f_4} & -\til{f_2}\\
\vphantom{\Big|} f_6 & f_6 & -2 f_7 & 0 & 0 & -2 f_5 & 2 f_4 & -2 f_0 & f_1 & \til{f_6} & 2 \til{f_7} & 0 & 0 & 2 \til{f_5} & -2 \til{f_4} & -2 \til{f_0} & -\til{f_1}\\
\hdashline
\vphantom{\Big|} f_7 & f_7 & f_6 & f_5 & f_4 & -f_3 & -f_2 & -f_1 & -f_0 & \til{f_7} & -\til{f_6} & -\til{f_5} & -\til{f_4} & \til{f_3} & \til{f_2} & \til{f_1} & -\til{f_0}\\
\hdashline
\vphantom{\Big|} \til{f_0} & \til{f_0} & -\til{f_1} & -\til{f_2} & -\til{f_3} & -\til{f_4} & -\til{f_5} & -\til{f_6} & -\til{f_7} & -f_0 & f_1 & f_2 & f_3 & f_4 & f_5 & f_6 & f_7\\
\hdashline
\vphantom{\Big|} \til{f_1} & \til{f_1} & 2 \til{f_0} & -2 \til{f_3} & 2 \til{f_2} & 0 & 0 & -2 \til{f_7} & \til{f_6} & -f_1 & -2 f_0 & -2 f_3 & 2 f_2 & 0 & 0 & -2 f_7 & f_6\\
\vphantom{\Big|} \til{f_2} & \til{f_2} & 2 \til{f_3} & 2 \til{f_0} & -2 \til{f_1} & 0 & -2 \til{f_7} & 0 & \til{f_5} & -f_2 & 2 f_3 & -2 f_0 & -2 f_1 & 0 & -2 f_7 & 0 & f_5\\
\vphantom{\Big|} \til{f_3} & \til{f_3} & -2 \til{f_2} & 2 \til{f_1} & 2 \til{f_0} & -2 \til{f_7} & 0 & 0 & \til{f_4} & -f_3 & -2 f_2 & 2 f_1 & -2 f_0 & -2 f_7 & 0 & 0 & f_4\\
\hdashline
\vphantom{\Big|} \til{f_4} & \til{f_4} & 0 & 0 & 2 \til{f_7} & 2 \til{f_0} & 2 \til{f_6} & -2 \til{f_5} & -\til{f_3} & -f_4 & 0 & 0 & 2 f_7 & -2 f_0 & 2 f_6 & -2 f_5 & -f_3\\
\vphantom{\Big|} \til{f_5} & \til{f_5} & 0 & 2 \til{f_7} & 0 & -2 \til{f_6} & 2 \til{f_0} & 2 \til{f_4} & -\til{f_2} & -f_5 & 0 & 2 f_7 & 0 & -2 f_6 & -2 f_0 & 2 f_4 & -f_2\\
\vphantom{\Big|} \til{f_6} & \til{f_6} & 2 \til{f_7} & 0 & 0 & 2 \til{f_5} & -2 \til{f_4} & 2 \til{f_0} & -\til{f_1} & -f_6 & 2 f_7 & 0 & 0 & 2 f_5 & -2 f_4 & -2 f_0 & -f_1\\
\hdashline
\vphantom{\Big|} \til{f_7} & \til{f_7} & -\til{f_6} & -\til{f_5} & -\til{f_4} & \til{f_3} & \til{f_2} & \til{f_1} & \til{f_0} & -f_7 & -f_6 & -f_5 & -f_4 & f_3 & f_2 & f_1 & -f_0\\
\hline
\end{array}
$
\caption{\label{table:sedenions-mult} Multiplication table for zero divisor basis elements of $\mathbb{S}$.}
\end{table}
\end{landscape}

\section{Commutativity graph of the sedenions} \label{section:commutativity-graph}

\begin{lemma} {\rm \cite[p. 438]{schafer}}
\leavevmode
\begin{enumerate}[(1)]
    \item If $n \leq 1$ then $C_{\A_n} = \A_n$, so the vertex set of $\Gamma_C(\A_n)$ is the empty set.
    \item If $n \geq 2$ then $C_{\A_n} = \mathbb{R}$, so the vertex set of $\Gamma_C(\A_n)$ is $P(\A_n \setminus \mathbb{R})$.
\end{enumerate}
\end{lemma}

\begin{lemma} {\rm \cite[Lemma 8.11]{our_anticomm} } \label{lemma:A_n-commutativity-through-orthogonality}
Let $x \in \A_n \setminus \{ 0 \}$, $\Re(x) = 0$.
\begin{enumerate} [(1)]
\item If $n(x)=0$ and $n \leq 3$ then $C_{\A_n}(x) = \mathbb{R} \oplus O_{\A_n}(x)$; \label{item:norm-zero}
\item If $n(x) \neq 0$ then $C_{\A_n}(x) = \mathbb{R} \oplus \mathbb{R}x \oplus O_{\A_n}(x)$. \label{item:norm-nonzero}
\end{enumerate}
\end{lemma}

\begin{corollary} \label{corollary:small-commutativity-component}
Let $x \in \Main_n \setminus \mathbb{R}$, $\Im(x) \notin Z(\Main_n)$. Then the connected component of $\Gamma_C(\Main_n)$ which contains $x$ is the complete graph on the vertex set $P(x + \mathbb{R})$.
\end{corollary}

\begin{proof}
It follows from Lemma~\ref{lemma:A_n-commutativity-through-orthogonality} that $C_{\Main_n}(x) = C_{\Main_n}(\Im(x)) = \Lin(1, \Im(x)) = \Lin(1,x)$. Clearly, $P(\Lin(1,x) \setminus \mathbb{R}) = P(x + \mathbb{R})$. The desired statement follows immediately.
\end{proof}

Due to Corollary~\ref{corollary:small-commutativity-component}, it is natural to give the following definition:

\begin{definition}
Let $\mathbb{R} + Z(\mathbb{S})$ denote the set of elements of $\mathbb{S}$ whose imaginary part is a zero divisor. $\Gamma_C^Z(\mathbb{S})$ is the subgraph of $\Gamma_C(\mathbb{S})$ on the vertex set $P(\mathbb{R} + Z(\mathbb{S}))$.
\end{definition}

\begin{lemma} \label{lemma:centralizer-product}
Let $(a,b) \in Z(\mathbb{S})$, $z \in \mathbb{O}$. Then there exists $(x,y) \in \Im(C_{\mathbb{S}}((a,b)))$ such that $xy = z$ if and only if $z \perp \Lin(1,a,b)$. 
\end{lemma}

\begin{proof}
Consider some $(c,d) \in O_\mathbb{S}((a,b))$, $(c,d) \neq 0$, and assume without loss of generality that $n(a) = n(b) = n(c) = n(d) = 1$. By Proposition~\ref{proposition:sedenions-orthonormal-system}, the elements $1,a,b,c,d,ab,ac,ad$ form an orthonormal system with respect to the inner product $\langle \cdot, \cdot \rangle$. Lemma~\ref{lemma:A_n-anticomm} implies that $a,b,c,d,ab,ac,ad$ anticommute pairwise. By Lemma~\ref{lemma:sedenions-zero-divisors-relation}, $ac = bd = -db$, $ad = -bc = cb$, $ab = -cd = (ac)(ad)$, so we obtain
\begin{align*}
    a(ad) &= -d = -(ad)a, & c(ad) &= c(cb) = -b = -(ad)c,\\
    a(ac) &= -c = -(ac)a, & c(ac) &= -c(ca) = a = -(ac)c,\\
    b(ad) &= -b(bc) = c = -(ad)b, & d(ad) &= -d(da) = a = -(ad)d,\\
    b(ac) &= b(bd) = -d = -(ac)b, & d(ac) &= -d(db) = b = -(ac)d.
\end{align*}
It follows from Lemma~\ref{lemma:A_n-commutativity-through-orthogonality} that
$$
\Im(C_{\mathbb{S}}((a,b))) = \mathbb{R}(a,b) \oplus O_{\mathbb{S}}((a,b)) = \Lin \left( (a,b), (c,d), (d,-c), (ad,ac), (ac,-ad) \right).
$$
Consider some $(x,y) \in \Im(C_{\mathbb{S}}((a,b)))$, $(x,y) = k_0(a,b) + k_1(c,d) + k_2(d,-c) + k_3(ad,ac) + k_4(ac,-ad)$. Then we have
\begin{align*}
x &= k_0 a + k_1 c + k_2 d + k_3 ad + k_4 ac,\\
y &= k_0 b - k_2 c + k_1 d - k_4 ad + k_3 ac,\\
xy &= 2k_0(k_1 ad - k_2 ac - k_3 c + k_4 d) - (k_1^2 + k_2^2 + k_3^2 + k_4^2 - k_0^2) ab.
\end{align*}
Clearly, $xy \perp \Lin(1,a,b)$. Moreover, for any $z \perp \Lin(1,a,b)$ we can choose $k_j$, $j = 0, \dots, 5$, such that $xy = z$. Let $z = l_1 ad - l_2 ac - l_3 c + l_4 d - l_0 ab$, $n(z) = l_0^2 + l_1^2 + l_2^2 + l_3^2 + l_4^2$. We have $n(xy) = n(x)n(y) = (k_0^2 + k_1^2 + k_2^2 + k_3^2 + k_4^2)^2$, so we need $k_0^2 + k_1^2 + k_2^2 + k_3^2 + k_4^2 = \sqrt{n(z)}$ and $k_1^2 + k_2^2 + k_3^2 + k_4^2 - k_0^2 = l_0$. Hence $k_0^2 = \frac{\sqrt{n(z)} - l_0}{2}$. Note that $l_0 \leq \sqrt{n(z)}$, so we can take $k_0 = \sqrt{\frac{\sqrt{n(z)} - l_0}{2}}$. Consider two cases:
\begin{itemize}
    \item If $l_0 = \sqrt{n(z)}$ then $k_0 = 0$. Furthermore, $l_1 = l_2 = l_3 = l_4 = 0$. Thus we can take any $k_1, k_2, k_3, k_4$ which satisfy $k_1^2 + k_2^2 + k_3^2 + k_4^2 = \sqrt{n(z)} = l_0$.
    \item If $l_0 < \sqrt{n(z)}$ then $k_0 > 0$. Hence we set $k_j = \frac{l_j}{2 k_0}$, $j = 1, \dots, 4$. Then 
    $$
    k_1^2 + k_2^2 + k_3^2 + k_4^2 = \frac{l_1^2 + l_2^2 + l_3^2 + l_4^2}{4 k_0^2} = \frac{n(z) - l_0^2}{4 k_0^2}.
    $$
    Thus we have 
    \begin{multline*}
    k_1^2 + k_2^2 + k_3^2 + k_4^2 - k_0^2 = \frac{n(z) - l_0^2 - 4k_0^4}{4 k_0^2} =\\
    = \frac{n(z) - l_0^2 - (\sqrt{n(z)} - l_0)^2}{2 (\sqrt{n(z)} - l_0)} = \frac{2 l_0 (\sqrt{n(z)} - l_0)}{2 (\sqrt{n(z)} - l_0)} = l_0.
    \end{multline*}
\end{itemize}
\end{proof}

\begin{proposition}
$\Gamma_C^Z(\mathbb{S})$ is connected, and its diameter is at most $4$.
\end{proposition}

\begin{proof}
Let $(a,b), (a',b') \in Z(\mathbb{S})$. Consider some $z \perp \Lin(1,a,b,a',b')$, $z \neq 0$. By Lemma~\ref{lemma:centralizer-product}, there exist $(c,d) \in \Im(C_{\mathbb{S}}((a,b)))$ and $(c',d') \in \Im(C_{\mathbb{S}}((a',b')))$ such that $cd = c'd' = z$. It follows from Corollary~\ref{corollary:small-commutativity-component} that $(c,d), (c',d') \in Z(\mathbb{S})$. By Theorem~\ref{theorem:component-octonion}, $(c,d)$ and $(c',d')$ belong to the same connected component of $\Gamma_O(\mathbb{S})$. Theorem~\ref{theorem:sedenions-orthogonality} implies that $d_{\Gamma_O(\mathbb{S})}((c,d),(c',d')) \leq 3$. However, it follows from Corollary~\ref{corollary:sedenions-path} that any path between $(c,d)$ and $(c',d')$ is of even length. Hence $d_{\Gamma_O(\mathbb{S})}((c,d),(c',d')) \leq 2$. Any path in $\Gamma_O(\mathbb{S})$ is also a path in $\Gamma_C(\mathbb{S})$, so $d_{\Gamma_C(\mathbb{S})}((c,d),(c',d')) \leq 2$. Then $d_{\Gamma_C(\mathbb{S})}((a,b),(a',b')) \leq 4$.
\end{proof}

\begin{proposition}
The diameter of $\Gamma_C^Z(\mathbb{S})$ is at least $3$.
\end{proposition}

\begin{proof}
Consider some $(a,b),(c,d) \in Z(\mathbb{S})$, $(a,b)(c,d) = 0$, $n(a) = n(b) = n(c) = n(d) = 1$. Then it follows from Figure~\ref{figure:double-hexagon} that 
\begin{align*}
O_{\mathbb{S}}((a,b)) &= \Lin((c,d), (d,-c), (ad,ac), (ac,-ad)),\\
O_{\mathbb{S}}((b,a)) &= \Lin((d,c), (c,-d), (ac,ad), (ad,-ac)).
\end{align*}
By Lemma~\ref{lemma:A_n-commutativity-through-orthogonality},
\begin{align*}
C_{\mathbb{S}}((a,b)) &= \Lin(1,(a,b)) \oplus O_{\mathbb{S}}((a,b)) = \Lin(1, (a,b), (c,d), (d,-c), (ad,ac), (ac,-ad)),\\
C_{\mathbb{S}}((b,a)) &= \Lin(1,(b,a)) \oplus O_{\mathbb{S}}((b,a)) = \Lin(1, (b,a), (d,c), (c,-d), (ac,ad), (ad,-ac)).
\end{align*}
Then it follows from Theorem~\ref{theorem:double-hexagon}\eqref{item:linear-independence} that $C_{\mathbb{S}}((a,b)) \cap C_{\mathbb{S}}((b,a)) = \mathbb{R}$, so $d_{\Gamma_C(\mathbb{S})}((a,b),(b,a)) \geq 3$.
\end{proof}

We have considered many examples by using Wolfram Mathematica, and in all cases we were able to find a path of length $3$ between two zero divisors. This leads us to the following conjecture.

\begin{conjecture}
The diameter of $\Gamma_C^Z(\mathbb{S})$ equals $3$.
\end{conjecture}

\end{document}